\documentclass[oneside,english]{amsart}
\usepackage[T1]{fontenc}
\usepackage[latin9]{inputenc}
\usepackage{prettyref}
\usepackage{amsthm}
\usepackage{amssymb}
\usepackage{nameref}

\makeatletter
\numberwithin{equation}{section}
\numberwithin{figure}{section}
\theoremstyle{plain}
\newtheorem{thm}{\protect\theoremname}[section]
\theoremstyle{plain}
\newtheorem{lem}[thm]{\protect\lemmaname}
\theoremstyle{plain}
\newtheorem{cor}[thm]{\protect\corollaryname}
\theoremstyle{remark}
\newtheorem{rem}[thm]{\protect\remarkname}
\theoremstyle{remark}
\newtheorem*{claim*}{\protect\claimname}
\theoremstyle{definition}
\newtheorem{defn}[thm]{\protect\definitionname}
\theoremstyle{definition}
\newtheorem{example}[thm]{\protect\examplename}

\usepackage{mathtools}
\usepackage{braket}
\usepackage{prettyref}
\usepackage{tikz}
\usetikzlibrary{patterns}

\newrefformat{prop}{Proposition \ref{#1}}
\newrefformat{cor}{Corollary \ref{#1}}
\newrefformat{exa}{Example \ref{#1}}

\global\long\def\ns#1{\prescript{\ast}{}{#1}}

\makeatother

\usepackage{babel}
\providecommand{\claimname}{Claim}
\providecommand{\corollaryname}{Corollary}
\providecommand{\definitionname}{Definition}
\providecommand{\examplename}{Example}
\providecommand{\lemmaname}{Lemma}
\providecommand{\remarkname}{Remark}
\providecommand{\theoremname}{Theorem}

\begin{document}
\title{Fehrele's principle in nonstandard topology}
\author{Takuma Imamura}
\address{Research Institute for Mathematical Sciences\\
Kyoto University\\
Kitashirakawa Oiwake-cho, Sakyo-ku, Kyoto 606-8502, Japan}
\email{timamura@kurims.kyoto-u.ac.jp}
\begin{abstract}
In nonstandard analysis, Fehrele's principle is a beautiful criterion
for a set to be internal, stating that every galactic halic set is
internal. In this note, we use this principle to prove some well-known
results in topology, including slight generalisations of the Moore-Osgood
theorem and Dini's theorem.
\end{abstract}

\maketitle

\section{Introduction}

Throughout the note, let $\mathbb{U}$ be a standard universe and
$\ns{\mathbb{U}}$ a sufficiently saturated elementary extension of
$\mathbb{U}$. A subset $G$ of $\ns{\mathbb{U}}$ is said to be \emph{galactic}
if there is a family $\mathcal{G}=\set{G_{i}|i\in S}$ of internal
sets, where $S$ is standard, such that $G=\bigcup\mathcal{G}$. A
subset $H$ of $\ns{\mathbb{U}}$ is said to be \emph{halic} if there
is a family $\mathcal{H}=\set{H_{i}|i\in S}$ of internal sets, where
$S$ is standard, such that $H=\bigcap\mathcal{H}$. Note that $\mathcal{G}$
and $\mathcal{H}$ themselves are not necessarily internal.

Fehrele's principle now can be stated as follows: every galactic halic
set is internal. The aim of this note is to demonstrate how to use
this principle by proving some theorems in topology. In \prettyref{sec:Fehrele's-principle},
we recall the statements and the proofs of Fehrele's principle and
some related principles. In \prettyref{sec:Some-applications-in-topology},
we prove some well-known theorems in topology, including slight generalisations
of the Moore-Osgood theorem and Dini's theorem.

We refer to Robinson \cite{Rob66} and Stroyan and Luxemburg \cite{SL76}
for nonstandard analysis, and van den Berg \cite{vdB87} for the details
of Fehrele's principle.

\section{\label{sec:Fehrele's-principle}Fehrele's principle}

Let $\varDelta$ be a directed set. An element $\delta\in\ns{\varDelta}$
is \emph{limited} if $\delta$ is dominated by some element of $\varDelta$;
and $\delta$ is \emph{illimited} if $\delta$ dominates $\varDelta$.
Note that when $\varDelta$ is not linearly ordered, $\ns{\varDelta}$
may have elements which are neither limited nor illimited; and when
$\varDelta$ is self-bounded, $\ns{\varDelta}$ has elements which
are both limited and illimited.
\begin{lem}[Overspill]
\label{lem:Overspill}Let $\varDelta$ be a directed set and $A$
an internal subset of $\ns{\varDelta}$.
\begin{enumerate}
\item If $A$ contains all sufficiently large limited elements of $\ns{\varDelta}$,
then it also contains all sufficiently small illimited elements of
$\ns{\varDelta}$.
\item If $A$ contains arbitrarily large limited elements of $\ns{\varDelta}$,
then it also contains arbitrarily small illimited elements of $\ns{\varDelta}$.
\end{enumerate}
\end{lem}

\begin{proof}
\begin{enumerate}
\item For $L\in\varDelta$ set $A_{L}:=\set{U\in\ns{\varDelta}|L\leq U\wedge\left[L,U\right]\subseteq A}$.
By assumption, the family $\set{A_{L}|L\in\varDelta,L\geq L_{0}}$
has the finite intersection property for some limited $L_{0}\in\ns{\varDelta}$.
Hence we can pick an element $U\in\bigcap_{L\in\varDelta,L\geq L_{0}}A_{L}$
by saturation. Every illimited element of $\ns{\varDelta}$ below
$U$ belongs to $A$.
\item Let $U\in\ns{\varDelta}$ be illimited. For $L\in\varDelta$ set $B_{L}:=\left[L,U\right]\cap A$.
By assumption, the family $\set{B_{L}|L\in\varDelta}$ has the finite
intersection property. Hence we can pick an element $\delta\in\bigcap_{L\in\varDelta}B_{L}$
by saturation. $\delta$ is an illimited element of $\ns{\varDelta}$
below $U$ and belongs to $A$.\qedhere
\end{enumerate}
\end{proof}
\begin{lem}[Prolongation]
\label{thm:Prolongation}Every map $f\colon S\to\ns{\mathbb{U}}$,
where $S$ is standard, can be extended to an internal map $\ns{S}\to\ns{\mathbb{U}}$.
\end{lem}

\begin{proof}
For each $i\in S$ let $A_{i}:=\set{g\in\ns{\mathbb{U}}|g\left(i\right)=f\left(i\right)}$.
It is easy to see that the family $\set{A_{i}|i\in S}$ has the finite
intersection property. By saturation, the intersection $\bigcap_{i\in S}A_{i}$
has an element $g$ which extends $f$ internally.
\end{proof}
\begin{thm}[Separation]
\label{thm:Separation}Let $G$ be a galactic set and $H$ a halic
set, and suppose that $G\subseteq H$. Then there exists an internal
set $I$ such that $G\subseteq I\subseteq H$.
\end{thm}

\begin{proof}
Take $\mathcal{G}$ and $\mathcal{H}$ such that $G=\bigcup\mathcal{G}$
and $H=\bigcap\mathcal{H}$. We may assume without loss of generality
that both $\mathcal{G}$ and $\mathcal{H}$ are indexed by the same
standard set $S$: $\mathcal{G}=\set{G_{i}|i\in S}$ and $\mathcal{H}=\set{H_{i}|i\in S}$.
Extend $\mathcal{G}$ and $\mathcal{H}$ to $\ns{S}$ internally by
\nameref{thm:Prolongation}.

The set $\mathcal{P}_{\mathrm{fin}}\left(S\right)$ of all finite
subsets of $S$ forms a directed set with respect to $\subseteq$.
Note that the illimited elements of $\ns{\left(\mathcal{P}_{\mathrm{fin}}\left(S\right)\right)}$
are precisely hyperfinite subsets of $\ns{S}$ containing $S$. Now
consider the internal set $\mathcal{A}:=\set{T\in\ns{\left(\mathcal{P}_{\mathrm{fin}}\left(S\right)\right)}|\bigcup_{i\in T}G_{i}\subseteq\bigcap_{i\in T}H_{i}}$.
Clearly $\mathcal{P}_{\mathrm{fin}}\left(S\right)\subseteq\mathcal{A}$.
By \nameref{lem:Overspill}, $\mathcal{A}$ includes an illimited
element $T$ of $\ns{\left(\mathcal{P}_{\mathrm{fin}}\left(S\right)\right)}$.
Hence $G=\bigcup_{i\in S}G_{i}\subseteq\bigcup_{i\in T}G_{i}\subseteq\bigcap_{i\in T}H_{i}\subseteq\bigcap_{i\in S}H_{i}=H$.
The inner two sets, $\bigcup_{i\in T}G_{i}$ and $\bigcap_{i\in T}H_{i}$,
are internal.
\end{proof}
\begin{thm}[Fehrele's principle]
\label{thm:Fehrele}Every galactic halic set is internal.
\end{thm}

\begin{proof}
Let $I$ be such a set. Obviously $I\subseteq I$. By \nameref{thm:Separation}
there is an internal set $J$ with $I\subseteq J\subseteq I$, so
$I=J$ is internal.
\end{proof}
\nameref{thm:Fehrele} (in the form of \nameref{thm:Separation})
tells us that \nameref{lem:Overspill} also holds for halic subsets.
\begin{cor}[Halic overspill]
\label{cor:Halic-Overspill}Let $\varDelta$ be a directed set and
$H$ a halic subset of $\ns{\varDelta}$.
\begin{enumerate}
\item If $H$ contains all sufficiently large limited elements of $\ns{\varDelta}$,
then it also contains all sufficiently small illimited elements of
$\ns{\varDelta}$.
\item If $H$ contains arbitrarily large limited elements of $\ns{\varDelta}$,
then it also contains arbitrarily small illimited elements of $\ns{\varDelta}$.
\end{enumerate}
\end{cor}

\begin{proof}
\begin{enumerate}
\item Fix a limited $L\in\ns{\varDelta}$ such that every limited element
of $\ns{\varDelta}$ above $L$ belongs to $H$, i.e., $\bigcup_{U\in\varDelta}\left[L,U\right]\subseteq H$.
By \nameref{thm:Separation}, there exists an internal set $I$ such
that $\bigcup_{U\in\varDelta}\left[L,U\right]\subseteq I\subseteq H$.
By \nameref{lem:Overspill}, $I$ has all sufficiently small illimited
elements, and so does $H$.
\item Fix a family $\set{H_{i}|i\in S}$ of internal sets indexed by a standard
set such that $H=\bigcap_{i\in S}H_{i}$. Let $U\in\ns{\varDelta}$
be illimited. The set defined by
\[
H':=\set{L\in\ns{\varDelta}|\left[L,U\right]\cap H\neq\varnothing}
\]
is halic, because
\begin{align*}
H' & =\set{L\in\ns{\varDelta}|\left[L,U\right]\cap\bigcap_{i\in S}H_{i}\neq\varnothing}\\
 & =\set{L\in\ns{\varDelta}|\forall T\in\mathcal{P}_{\mathrm{fin}}\left(S\right).\left[L,U\right]\cap\bigcap_{i\in T}H_{i}\neq\varnothing}\\
 & =\bigcap_{T\in\mathcal{P}_{\mathrm{fin}}\left(S\right)}\set{L\in\ns{\varDelta}|\left[L,U\right]\cap H_{i}\neq\varnothing},
\end{align*}
where the second equality is by saturation. Every limited element
of $\prescript{\ast}{}{\varDelta}$ belongs to $H'$ by assumption.
By (1), $H'$ has an illimited element $L$ of $\ns{\varDelta}$,
which satisfies that $\left[L,U\right]\cap H\neq\varnothing$. It
follows that $H$ has an illimited element of $\ns{\varDelta}$ below
$U$ (and above $L$).\qedhere
\end{enumerate}
\end{proof}
\begin{cor}[Galactic underspill]
\label{cor:Galactic-Underspill}Let $\varDelta$ be a directed set
and $G$ a galactic subset of $\ns{\varDelta}$.
\begin{enumerate}
\item If $G$ contains all sufficiently small illimited elements of $\ns{\varDelta}$,
then it also contains all sufficiently large limited elements of $\ns{\varDelta}$.
\item If $G$ contains arbitrarily small illimited elements of $\ns{\varDelta}$,
then it also contains arbitrarily large limited elements of $\ns{\varDelta}$.
\end{enumerate}
\end{cor}

\begin{proof}
Apply the contraposition of \nameref{cor:Halic-Overspill} to the
complement $H:=\ns{\varDelta}\setminus G$.
\end{proof}

\section{\label{sec:Some-applications-in-topology}Some applications in topology}

The following is a generalisation of \cite[Theorem 4.3.10]{Rob66}
to nets in uniform spaces. In most applications to (small-scale) topology,
\nameref{thm:Fehrele} will be appeared in this form.
\begin{lem}[Robinson's lemma]
\label{lem:Robinson}Let $X$ be a uniform space, $\varDelta$ a
directed set, and $\set{x_{\delta}}_{\delta\in\ns{\varDelta}}$ and
$\set{y_{\delta}}_{\delta\in\ns{\varDelta}}$ internal nets in $\ns{X}$.
\begin{enumerate}
\item If $x_{\delta}\approx_{X}y_{\delta}$ holds for all sufficiently large
limited $\delta\in\ns{\varDelta}$, then it also holds for all sufficiently
small illimited $\delta\in\ns{\varDelta}$.
\item If $x_{\delta}\approx_{X}y_{\delta}$ holds for arbitrarily large
limited $\delta\in\ns{\varDelta}$, then it also holds for arbitrarily
small illimited $\delta\in\ns{\varDelta}$.
\end{enumerate}
\end{lem}

\begin{proof}
Apply \nameref{cor:Halic-Overspill} to $H:=\set{\delta\in\ns{\varDelta}|x_{\delta}\approx_{X}y_{\delta}}$.
\end{proof}
\begin{rem}
When $X$ is a pseudometric space and $\varDelta=\mathbb{N}$, this
lemma can be proved only using \nameref{lem:Overspill} for internal
sets. Consider the internal set
\[
I:=\set{n\in\ns{\mathbb{N}}|\ns{d_{X}}\left(x_{n},y_{n}\right)\leq2^{-n}}.
\]
By assumption, $I$ has all sufficiently (resp. arbitrarily) large
limited elements. By \nameref{lem:Overspill}, $I$ has all sufficiently
(resp. arbitrarily) small illimited elements. Clearly, for each illimited
$n\in I$, we have that $x_{n}\approx_{X}y_{n}$, because $\ns{d_{X}}\left(x_{n},y_{n}\right)\leq2^{-n}\approx_{\mathbb{R}}0$.
\end{rem}

Recall that two Cauchy nets $\set{x_{\delta}}_{\delta\in\varDelta}$
and $\set{y_{\gamma}}_{\gamma\in\varGamma}$ in a uniform space $X$
are said to be \emph{equivalent} if for each $U\in\mathcal{U}_{X}$
there are $\delta_{0}\in\varDelta$ and $\gamma_{0}\in\varGamma$
such that $\left(x_{\delta},y_{\gamma}\right)\in U$ holds for all
$\delta\geq\delta_{0}$ and $\gamma\geq\gamma_{0}$. This notion has
the following simple nonstandard characterisation:
\begin{lem}
\label{lem:Equivalence}Let $\set{x_{\delta}}_{\delta\in\varDelta}$
and $\set{y_{\gamma}}_{\gamma\in\varGamma}$ nets in a uniform space
$X$. The following are equivalent:
\begin{enumerate}
\item $\set{x_{\delta}}_{\delta\in\varDelta}$ and $\set{y_{\gamma}}_{\gamma\in\varGamma}$
are equivalent Cauchy nets;
\item for any illimited $\delta$ and $\gamma$, we have $\ns{x_{\delta}}\approx_{X}\ns{y_{\gamma}}$.
\end{enumerate}
\end{lem}

\begin{proof}
Easy.
\end{proof}
If one of equivalent Cauchy nets is convergent, then the others are
also convergent to the same point. This well-known fact can easily
be shown by combining the nonstandard characterisations of convergence
and equivalence, so the proof is left to the reader.
\begin{lem}
\label{lem:Product-Directed-Set}Let $\varDelta$ and $\varGamma$
be directed sets and $\left(\delta,\gamma\right)\in\ns{\varDelta}\times\ns{\varGamma}$.
\begin{enumerate}
\item $\left(\delta,\gamma\right)$ is limited if and only if both $\delta$
and $\gamma$ are limited.
\item $\left(\delta,\gamma\right)$ is illimited if and only if both $\delta$
and $\gamma$ are illimited.
\end{enumerate}
\end{lem}

\begin{proof}
Easy.
\end{proof}
Notice that if either $\delta$ or $\gamma$ is limited and the other
is illimited, then $\left(\delta,\gamma\right)$ is neither limited
nor illimited.
\begin{thm}
\label{thm:weak-Moore-Osgood}Let $X$ be a uniform space, $\varDelta$
and $\varGamma$ directed sets, $f\colon\varDelta\times\varGamma\to X$,
$g\colon\varGamma\to X$ and $h\colon\varDelta\to X$. Suppose that
\begin{enumerate}
\item $\set{f\left(\delta,\cdot\right)}_{\delta\in\varDelta}$ is uniformly
convergent to $g$; and
\item $\set{f\left(\cdot,\gamma\right)}_{\gamma\in\varGamma}$ is uniformly
convergent to $h$.
\end{enumerate}
Then $f,g,h$ are equivalent Cauchy nets. Hence if one of $f,g,h$
is convergent to some point, then the other two are also convergent
to the same point.

\end{thm}

\begin{proof}
By the nonstandard characterisation of uniform convergence, $\ns{g}\left(\gamma\right)\approx_{X}\ns{f}\left(\delta,\gamma\right)\approx_{X}\ns{h}\left(\delta\right)$
holds for all illimited elements $\delta$ and $\gamma$ of $\ns{\varDelta}$
and $\ns{\varGamma}$, respectively. It follows that $\ns{g}\left(\gamma\right)\approx\ns{f}\left(\delta',\gamma'\right)\approx\ns{h}\left(\delta\right)$
for all illimited $\delta,\delta',\gamma,\gamma'$. Hence $f,g,h$
are equivalent by \prettyref{lem:Equivalence} and \prettyref{lem:Product-Directed-Set}.
\end{proof}
The proof above is quite easy and does not depend on saturation (only
depends on enlargement). On the other hand, the following theorem
(cf. \cite[pp. 593--594]{Osg12}) depends on saturation.
\begin{thm}[Moore-Osgood theorem]
\label{thm:Moore-Osgood}Let $X$ be a uniform space, $\varDelta$
and $\varGamma$ directed sets, $f\colon\varDelta\times\varGamma\to X$,
$g\colon\varGamma\to X$ and $h\colon\varDelta\to X$. Suppose that
\begin{enumerate}
\item $\set{f\left(\delta,\cdot\right)}_{\delta\in\varDelta}$ is uniformly
convergent to $g$; and
\item $\set{f\left(\cdot,\gamma\right)}_{\gamma\in\varGamma}$ is pointwise
convergent to $h$.
\end{enumerate}
Then $f$ and $g$ are equivalent Cauchy nets, and $h$ is Cauchy.
Moreover, if $h$ is convergent to some point, then $f$ and $g$
are also convergent to the same point. Furthermore, if either $f$
or $g$ is convergent to some point, then $h$ is also convergent
to the same point. Hence if one of $f,g,h$ is convergent to some
point, then the other two are also convergent to the same point.

\end{thm}

\begin{proof}
Applying the nonstandard characterisations of convergences (\cite[Theorem 4.2.4]{Rob66}
for pointwise convergence and \cite[Theorem 4.6.1]{Rob66} for uniform
convergence), we obtain:
\begin{enumerate}
\item[(1')] $\ns{f}\left(\delta,\gamma\right)\approx_{X}\ns{g}\left(\gamma\right)$
for all $\gamma\in\ns{\varGamma}$ and for all illimited $\delta\in\ns{\varDelta}$;
and
\item[(2')] $\ns{f}\left(\delta,\gamma\right)\approx_{X}h\left(\delta\right)$
for all $\delta\in\varDelta$ and for all illimited $\gamma\in\ns{\varGamma}$.
\end{enumerate}
\begin{claim*}
$g$ is Cauchy.
\end{claim*}

\subsubsection*{Proof}

Let $\gamma,\gamma'$ be illimited elements of $\ns{\varGamma}$.
By (2'), $\ns{f}\left(\delta,\gamma\right)\approx_{X}h\left(\delta\right)\approx_{X}\ns{f}\left(\delta,\gamma'\right)$
holds for all $\delta\in\varDelta$. By \nameref{lem:Robinson}, there
exists an illimited $\delta\in\ns{\varDelta}$ such that $\ns{f}\left(\delta,\gamma\right)\approx_{X}\ns{f}\left(\delta,\gamma'\right)$,
so $\ns{g}\left(\gamma\right)\approx_{X}\ns{f}\left(\delta,\gamma\right)\approx_{X}\ns{f}\left(\delta,\gamma'\right)\approx_{X}\ns{g}\left(\gamma'\right)$
by (1'). Hence $g$ is Cauchy by the nonstandard characterisation
of Cauchiness \cite[Theorem 8.4.26]{SL76}.
\begin{claim*}
$f$ is a Cauchy net equivalent to $g$.
\end{claim*}

\subsubsection*{Proof }

Since $g$ is Cauchy, $\ns{f}\left(\delta,\gamma\right)\approx_{X}\ns{g}\left(\gamma\right)\approx_{X}\ns{g}\left(\gamma'\right)$
holds for all illimited $\gamma,\gamma'\in\ns{\varGamma}$ and illimited
$\delta\in\ns{\varDelta}$. Hence $f$ and $g$ are equivalent Cauchy
nets by \prettyref{lem:Equivalence} and \prettyref{lem:Product-Directed-Set}.
\begin{claim*}
If $h$ is convergent to $x\in X$, then so is $g$.
\end{claim*}

\subsubsection*{Proof}

The proof is similar to the first claim. Let $\gamma\in\ns{\varGamma}$
be illimited. By (2'), $\ns{f}\left(\delta,\gamma\right)\approx_{X}h\left(\delta\right)$
holds for all $\delta\in\varDelta$. By \nameref{lem:Robinson}, $\ns{f}\left(\delta,\gamma\right)\approx_{X}\ns{h}\left(\delta\right)$
for some illimited $\delta\in\ns{\varDelta}$. Fix a witness $\delta$.
Since $h$ is convergent to $x$, $\ns{g}\left(\gamma\right)\approx_{X}\ns{f}\left(\delta,\gamma\right)\approx_{X}\ns{h}\left(\delta\right)\approx_{X}x$
by (1'). Hence $g$ is also convergent to $x$.
\begin{claim*}
If $f$ is convergent to $x\in X$, then $h$ clusters at $x$.
\end{claim*}

\subsubsection*{Proof}

Fix an illimited $\gamma\in\ns{\varGamma}$. Similarly to the previous
claim, we establish that $\ns{f}\left(\delta,\gamma\right)\approx_{X}h\left(\delta\right)$
for some illimited $\delta\in\ns{\varDelta}$ by \nameref{lem:Robinson}.
For such $\delta$ we have that $\ns{h}\left(\delta\right)\approx_{X}\ns{f}\left(\delta,\gamma\right)\approx_{X}x$.
Hence $h$ clusters at $x$ by the nonstandard characterisation of
cluster points \cite[Theorem 4.2.5]{Rob66}.
\begin{claim*}
$h$ is Cauchy.
\end{claim*}

\subsubsection*{Proof}

Unfortunately the proof of this claim is not purely nonstandard. Let
$\delta,\delta'\in\ns{\varDelta}$ be illimited. By saturation, we
can find a $U\in\ns{\mathcal{U}_{X}}$ such that $U\subseteq{\approx_{X}}$.
Applying transfer to (2), we have that $\ns{h}\left(\delta\right)\mathrel{U}\ns{f}\left(\delta,\gamma\right)$
and $\ns{h}\left(\delta'\right)\mathrel{U}\ns{f}\left(\delta',\gamma\right)$
for all sufficiently large $\gamma\in\ns{\varGamma}$. (Here we used
the $\varepsilon\delta$-style definition of pointwise convergence.)
We fix such $\gamma$. Obviously $\ns{h}\left(\delta\right)\approx_{X}\ns{f}\left(\delta,\gamma\right)$
and $\ns{h}\left(\delta'\right)\approx_{X}\ns{f}\left(\delta',\gamma\right)$.
By (1'), $\ns{f}\left(\delta,\gamma\right)\approx_{X}\ns{g}\left(\gamma\right)$
and $\ns{f}\left(\delta',\gamma\right)\approx_{X}\ns{g}\left(\gamma\right)$.
By combining them, we have that $\ns{h}\left(\delta\right)\approx_{X}\ns{h}\left(\delta'\right)$.
Hence $h$ is Cauchy.\qedhere

\end{proof}
\begin{thm}
\label{thm:Uniform-Convergence}Let $X$ be a topological (resp. uniform)
space, $Y$ a uniform spaces, $\varDelta$ a directed set, $\set{f_{\delta}\colon X\to Y}_{\delta\in\varDelta}$
a net of continuous (resp. uniformly continuous) maps, and $g\colon X\to Y$.
If $\set{f_{\delta}}_{\delta\in\varDelta}$ is uniformly convergent
to $g$, then $g$ is continuous (resp. uniformly continuous).
\end{thm}

This theorem can be thought of as a corollary to the Moore--Osgood
theorem: let $\set{x_{\gamma}}_{\gamma\in\varGamma}$ be a net in
$X$ converging to $x\in X$. We then obtain three maps $F\colon\varDelta\times\varGamma\to X$,
$G\colon\varGamma\to X$, $H\colon\varDelta\to X$ defined by $F\left(\delta,\gamma\right)=f_{\delta}\left(x_{\gamma}\right)$,
$G\left(\gamma\right)=g\left(x_{\gamma}\right)$, $H\left(\delta\right)=f_{\delta}\left(x\right)$.
Since $\set{f_{\delta}}_{\delta\in\varDelta}$ is uniformly convegent
to $g$, $\set{F\left(\delta,-\right)}_{\delta\in\varDelta}$ is uniformly
convergent to $G$. Similarly, since each $f_{\delta}$ is continuous
at $x$, $\set{F\left(-,\gamma\right)}_{\gamma\in\varGamma}$ is convergent
to $H$. Moreover, $H$ is convergent to $g\left(x\right)$. Thanks
to \prettyref{thm:Moore-Osgood}, $G$ is also convergent to the same
point(s) as $H$, and therefore $\lim_{\gamma\in\varGamma}g\left(x_{\gamma}\right)=\lim_{\gamma\in\varGamma}G\left(\gamma\right)=\lim_{\delta\in\varDelta}H\left(\delta\right)=g\left(x\right)$.
The same applies to the uniform case. (Replace $x_{\gamma}$ and $x$
with maps of the form $Z\to X$, and assume $\set{x_{\gamma}}_{\gamma\in\varGamma}$
is uniformly convergent to $x$. Then apply \prettyref{thm:weak-Moore-Osgood}.)
We also give a direct proof which is simpler than the above one.
\begin{proof}
Let $x\in X$ (resp. $x\in\ns{X}$) and $y\in\mu_{X}\left(x\right)$.
For each $\delta\in\varDelta$, by the nonstandard characterisation
of continuity \cite[Theorem 4.2.7]{Rob66} (resp. uniform continuity
\cite[Theorem 8.4.23]{SL76}), it follows that $\ns{f_{\delta}}\left(x\right)\approx_{Y}\ns{f_{\delta}}\left(y\right)$.
Note that $\set{\ns{f_{\delta}}\left(x\right)}_{\delta\in\ns{\varDelta}}$
and $\set{\ns{f_{\delta}}\left(y\right)}_{\delta\in\ns{\varDelta}}$
are internal nets in $\ns{Y}$. By \nameref{lem:Robinson}, $\ns{f_{\delta}}\left(x\right)\approx_{Y}\ns{f_{\delta}}\left(y\right)$
holds also for some illimited $\delta\in\ns{\varDelta}$. By the nonstandard
characterisation of uniform convergence, we have that $\ns{g}\left(x\right)\approx_{Y}\ns{f_{\delta}}\left(x\right)\approx_{Y}\ns{f_{\delta}}\left(y\right)\approx_{Y}\ns{g}\left(y\right)$.
Since $x$ was arbitrary, $g$ is continuous (resp. uniformly continuous).
\end{proof}
\begin{defn}
Let $X$ be a set and $Y$ a uniform space. We say that a net $\set{f_{\delta}\colon X\to Y}_{\delta\in\varDelta}$
is \emph{monotonically convergent} to a map $g\colon X\to Y$ if
\begin{enumerate}
\item $\set{f_{\delta}}_{\delta\in\varDelta}$ is pointwise convergent to
$g$; and
\item there is a uniform base $\mathcal{B}_{Y}\subseteq\mathcal{U}_{Y}$
such that $f_{\gamma}\left(x\right)\mathrel{V}g\left(x\right)$ implies
$f_{\delta}\left(x\right)\mathrel{V}g\left(x\right)$ for all $x\in X$,
$V\in\mathcal{B}$, $\gamma,\delta\in\varDelta$ with $\gamma\leq\delta$.
\end{enumerate}
\end{defn}

\begin{example}
Let $Y=\mathbb{R}$ and $\varDelta=\mathbb{N}$. If $\set{f_{n}\colon X\to\mathbb{R}}_{n\in\mathbb{N}}$
is pointwise convergent to $g\colon X\to\mathbb{R}$, and if $\set{f_{n}\left(x\right)}_{n\in\mathbb{N}}$
is non-increasing or non-decreasing for each $x\in X$, then $\set{f_{n}}_{n\in\mathbb{N}}$
is monotonically convergent to $g$. The converse does not hold, because
the monotonicity condition in that definition only requires $\left|f_{n}-g\right|$
to be non-increasing.
\end{example}

\begin{lem}
\label{lem:Monotonicity}Suppose that the condition (2) holds. For
any $x\in\ns{X}$ and $\gamma,\delta\in\ns{\varDelta}$ with $\gamma\leq\delta$,
if $\ns{f_{\gamma}}\left(x\right)\approx_{Y}\ns{g}\left(x\right)$,
then $\ns{f_{\delta}}\left(x\right)\approx_{Y}\ns{g}\left(x\right)$.
\end{lem}

\begin{proof}
Suppose $\ns{f_{\gamma}}\left(x\right)\approx_{Y}\ns{g}\left(x\right)$
and $\gamma\leq\delta$. Let $V\in\mathcal{B}_{Y}$. Clearly $\ns{f_{\gamma}}\left(x\right)\mathrel{\ns{V}}\ns{g}\left(x\right)$
holds. By the transferred condition (2), it follows that $\ns{f_{\delta}}\left(x\right)\mathrel{\ns{V}}\ns{g}\left(x\right)$.
Hence $\ns{f_{\delta}}\left(x\right)\approx_{Y}\ns{g}\left(x\right)$,
because $V$ was arbitrary.
\end{proof}
\begin{thm}[Generalised Dini's theorem]
Let $X$ be a topological space, $Y$ a uniform space, $\varDelta$
a directed set, $\set{f_{\delta}\colon X\to Y}_{\delta\in\varDelta}$
a net of continuous maps, and $g\colon X\to Y$ a continuous map.
If $\set{f_{\delta}}_{\delta\in\varDelta}$ is monotonically convergent
to $g$, then the convergence is uniform on each compact subset of
$X$.
\end{thm}

\begin{proof}
Let $K$ be a compact subset of $X$. Let $x\in\ns{K}$. By the nonstandard
characterisation of compactness \cite[Corollary 4.1.15]{Rob66}, there
is a standard point $\prescript{\circ}{}{x}\in K$ such that $x\in\mu_{X}\left(\prescript{\circ}{}{x}\right)$.
For each $\delta\in\varDelta$, by the nonstandard characterisation
of continuity, $\ns{f_{\delta}}\left(x\right)\approx_{Y}f_{\delta}\left(\prescript{\circ}{}{x}\right)$
holds. Similarly $\ns{g}\left(x\right)\approx_{Y}g\left(\prescript{\circ}{}{x}\right)$
holds. Applying \nameref{lem:Robinson} to the former $\approx_{Y}$,
$\ns{f_{L}}\left(x\right)\approx_{Y}\ns{f_{L}}\left(\prescript{\circ}{}{x}\right)$
holds also for arbitrarily small illimited $L\in\ns{\varDelta}$.
Fix such $L$. By the nonstandard characterisation of (pointwise)
convergence, it follows that $\ns{f_{L}}\left(\prescript{\circ}{}{x}\right)\approx_{Y}g\left(\prescript{\circ}{}{x}\right)$.
Combining those $\approx_{Y}$, we have that $\ns{f_{L}}\left(x\right)\approx_{Y}\ns{g}\left(x\right)$.
By \prettyref{lem:Monotonicity}, $\ns{f_{\delta}}\left(x\right)\approx_{Y}\ns{g}\left(x\right)$
holds for any $\delta\geq L$. Hence $\ns{f_{\gamma}}\left(x\right)\approx_{Y}\ns{g}\left(x\right)$
holds also for all illimited $\delta\in\ns{\varDelta}$, because $L$
was arbitrarily small. By the nonstandard characterisation of uniform
convergence, $\set{f_{\delta}}_{\delta\in\varDelta}$ is uniformly
convergent to $g$ on $K$.
\end{proof}
\begin{rem}
One can weaken the notion of monotone convergence by replacing the
condition (2) with
\begin{enumerate}
\item[(2')] there exist a uniform base $\mathcal{B}_{Y}\subseteq\mathcal{U}_{Y}$
and a map $F\colon\mathcal{B}_{Y}\to\mathcal{U}_{Y}$ such that
\begin{enumerate}
\item $f_{\gamma}\left(x\right)\mathrel{V}g\left(x\right)$ implies $f_{\delta}\left(x\right)\mathrel{F\left(V\right)}g\left(x\right)$
for all $x\in X$, $V\in\mathcal{B}_{Y}$, $\gamma,\delta\in\varDelta$
with $\gamma\leq\delta$; and
\item the image $F\left(\mathcal{B}_{Y}\right)$ is a uniform base of $Y$.
\end{enumerate}
\end{enumerate}
Then \prettyref{lem:Monotonicity} remains true under this modification.
Thus the Dini's theorem holds also for this weakened monotone convergence.
Other generalised Dini's theorems can be found in Naimpally and Tikoo
\cite{NT90}, Kupka \cite{Kup97} and many other literature.
\end{rem}

\section{Discussion}

The use of Fehrele's principle can be avoided. For example, in the
proof of \prettyref{thm:Uniform-Convergence}, one can avoid the use
of \nameref{lem:Robinson} as follows. As in the first proof, $\ns{f_{\delta}}\left(x\right)\approx_{Y}\ns{f_{\delta}}\left(y\right)$
holds for all $\delta\in\varDelta$. Now let $U\in\mathcal{U}_{X}$.
Choose another $V\in\mathcal{U}_{X}$ such that $V^{3}\subseteq V$
and $V^{-1}=V$. Since $\set{f_{\delta}}_{\delta\in\varDelta}$ is
uniformly convergent to $g$, there exists a $\delta\in\varDelta$
such that $\ns{f_{\delta}}\left(x\right)\mathrel{\ns{V}}\ns{g}\left(x\right)$
and $\ns{f_{\delta}}\left(y\right)\mathrel{\ns{V}}\ns{g}\left(y\right)$
hold by transfer. (Here we used the $\varepsilon\delta$-style definition
of uniform convergence.) Fix such $\delta$. Clearly $\ns{f_{\delta}}\left(x\right)\mathrel{\ns{V}}\ns{f_{\delta}}\left(y\right)$
holds. Composing them, we have that $\ns{g}\left(x\right)\mathrel{\ns{U}}\ns{g}\left(y\right)$.
Hence $\ns{g}\left(x\right)\approx_{Y}\ns{g}\left(y\right)$. The
rest of the proof is the same as the first one. See also \cite[Theorem 4.6.2]{Rob66}
for the case where $X$ is metrisable and $\varDelta=\mathbb{N}$.
Although this alternative proof avoids the use of Fehrele's principle,
it is contaminated by an $\varepsilon\delta$-argument instead. Fehrele's
principle enables us to prove (complicated) theorems purely nonstandardly.

\bibliographystyle{amsplain}
\bibliography{bibtexsource}

\end{document}